\documentclass[a4paper,leqno,12pt]{amsart}
\usepackage{epsfig}
\usepackage{amsmath}
 \usepackage[pdftex]{hyperref}
\usepackage{amsfonts}
\usepackage{amsmath,amscd}
\usepackage{amssymb}
\usepackage{amsthm}
\usepackage{mathrsfs}
\usepackage[mathscr]{euscript}
\usepackage{graphicx}
\usepackage{xypic}

\theoremstyle{plain}
\newtheorem{definition}{Definition}[section]
\newtheorem{theorem}[definition]{Theorem}
\newtheorem{remark}[definition]{Remark}
\newtheorem{final Remarks}[definition]{Final Remarks}
\newtheorem{lemma}[definition]{Lemma}
\newtheorem{question}[definition]{Question}

\numberwithin{equation}{section}
\begin{document}
\title{Farrell-Jones spheres and inertia groups of complex projective spaces}
\vspace{2cm}

\author{Ramesh Kasilingam}

\email{rameshkasilingam.iitb@gmail.com  ; mathsramesh1984@gmail.com  }
\address{Statistics and Mathematics Unit,
Indian Statistical Institute,
Bangalore Centre, Bangalore - 560059, Karnataka, India.}

\date{}
%\paragraph{Classification.}
%57R55; 57R50.

\maketitle
\begin{abstract}
We introduce and study a new class of homotopy spheres called Farrell-Jones spheres. Using Farrell-Jones sphere we construct examples of closed negatively curved manifolds $M^{2n}$, where $n=7$ or $8$, which are homeomorphic but not diffeomorphic to complex hyperbolic manifolds, thereby giving a partial answer to a question raised by C.S. Aravinda and F.T. Farrell. We show that every exotic sphere not bounding a spin manifold (Hitchin sphere) is a Farrell-Jones sphere. We also discuss the relationship between inertia groups of $\mathbb{C}\mathbb{P}^n$ and Farrell-Jones spheres.    
\end{abstract}

\section{Introduction}
Let $\Theta_m$ be the group of homotopy spheres defined by M. Kervaire and J. Milnor in \cite{KM63}.
\begin{definition}\label{farrell}\rm{
We call $\Sigma^{2n}\in \Theta_{2n}$~$(n\geq4)$ a Farrell-Jones sphere if $\mathbb{C}\mathbb{P}^{n}\#\Sigma^{2n}$ is not concordant to $\mathbb{C}\mathbb{P}^{n}$.}
\end{definition}
The following theorem gives an equivalent definition of Farrell-Jones spheres, which we prove in Section 3:
\begin{theorem}\label{condiff}
 Let $\Sigma^{2n}$ be an exotic sphere in $\Theta_{2n}$~$(n\geq4)$.\\
$\Sigma^{2n}$ is a Farrell-Jones sphere if and only if $\mathbb{C}\mathbb{P}^{n}\#\Sigma^{2n}$ is not diffeomorphic to $\mathbb{C}\mathbb{P}^{n}$.
\end{theorem}
By \cite[Lemma 3.17]{FJ94}, there exists a Farrell-Jones sphere $\Sigma^{m}\in \Theta_{m}$ for all $m=8n+2$ $(n\geq 1)$ and for $m=8$. Also we prove the following theorem in Section 3:
\begin{theorem}\label{exotho}
The non zero element of  $\Theta_{2n}\cong \mathbb{Z}_2$ $(n=7~or~8)$ is a Farrell-Jones sphere.
 \end{theorem}
The study of Farrell-Jones spheres is motivated by the following result, which is a slight modification of \cite[Theorem 3.20]{FJ94}:
\begin{theorem}\label{mainthe}
Let $M^{2n}$ be any closed complex hyperbolic manifold of complex dimension $n$. Let $\Sigma^{2n}\in \Theta_{2n}$ be a Farrell-Jones sphere. Given a positive real number $\epsilon$, there exists a  finite sheeted cover $\mathcal{N}^{2n}$  of $M^{2n}$ such that the following is true for any finite sheeted cover $N^{2n}$  of $\mathcal{N}^{2n}$.
\begin{enumerate}
\item[\rm{(i)}] The smooth manifold $N^{2n}$ is not diffeomorphic to $N^{2n}\#\Sigma^{2n}$.
\item[\rm{(ii)}] The connected sum  $N^{2n}\#\Sigma^{2n}$ supports a negatively curved Riemannian metric whose sectional curvatures all lie in the closed interval $[-4 -\epsilon, -1 +\epsilon]$.
\end{enumerate}
\end{theorem}
The proof of the above Theorem \ref{mainthe} follows from \cite[Corollary 3.14 and Proposition 3.19]{FJ94}. By using Theorem \ref{exotho} and Theorem \ref{mainthe}, we also construct in section 2 examples of closed negatively curved manifolds $M^{2n}$, where $n=7$ or $8$, which are homeomorphic but not diffeomorphic to complex hyperbolic manifolds, thereby giving a partial answer to a question raised by C.S. Aravinda and F.T. Farrell \cite{AF04}.\\

Another source for Farrell-Jones spheres is the class of so called Hitchin spheres. In \cite{Hit74}, Hitchin showed that if $\Sigma$ is a homotopy sphere with a metric of positive scalar curvature, then $\alpha(\Sigma)=0$, where $\alpha:\Omega_{*}^{spin}\to KO_{*}$ is the ring homomorphism which associates to a spin bordism class the $KO$-valued index of the Dirac operator of a representative spin manifold. The following definition can be found in \cite [Remark 3.4]{MA}:
\begin{definition}\label{HISP}\rm{
 An exotic sphere $\Sigma^{m}\in \Theta_{m}$~$(m\geq1)$ is called a Hitchin sphere if $\alpha(\Sigma^{m})\neq 0$.}
\end{definition}
We prove the following theorem in Section 3:
\begin{theorem}\label{lemhit}
Every Hitchin $(8n+2)$-sphere $(n \geq1)$ is a Farrell-Jones sphere.
\end{theorem}
Recall that the collection of homotopy spheres which admit an orientation preserving diffeomorphism $M\to M\#\Sigma$ form the inertia group of $M$, denoted by $I(M)$. There is a canonical topological identification $\iota : M\to M\#\Sigma$ which is the identity outside of the attaching region; the subset of the inertia group consisting of homotopy spheres that admit a diffeomorphism homotopic to $\iota$ is called the homotopy inertia group $I_h(M)$. Similarly, the concordance inertia group of $M^m$, $I_c(M^m)\subseteq \Theta_m$, consists of those homotopy spheres $\Sigma^m$ such that $M$ and $M\#\Sigma^m$ are concordant. By Theorem \ref{condiff}, we have that  $\Sigma^{2n}$ is a Farrell-Jones sphere iff $\Sigma^{2n}\notin I(\mathbb{C}\mathbb{P}^{n})$ iff $\Sigma^{2n}\notin I_{c}(\mathbb{C}\mathbb{P}^{n})$ 
iff $\Sigma^{2n}\notin I_{h}(\mathbb{C}\mathbb{P}^{n})$. In section 4, we discuss the group $I(\mathbb{C}\mathbb{P}^{4n+1})$.
\section{ Exotic Smooth Structures on Complex Hyperbolic Manifolds}
The negatively curved Riemannian symmetric spaces are of four types: $\mathbb{R}\textbf{H}^n$, $\mathbb{C}\textbf{H}^n$, $\mathbb{H}\textbf{H}^n$ and  $\mathbb{O}\textbf{H}^2$, where $\mathbb{R}$, $\mathbb{C}$, $\mathbb{H}$, $\mathbb{O}$ denote the real, complex, quaternion and Cayley numbers, i.e., the four division algebras $\mathbb{K}$ over the real numbers whose dimensions over $\mathbb{R}$ are $d = 1$, 2, 4 and 8 respectively. A Riemannian manifold $M^{dn}$ is called a real, complex, quarternionic or Cayley hyperbolic manifold provided its universal cover is isometric to $\mathbb{R}\textbf{H}^n$, $\mathbb{C}\textbf{H}^n$, $\mathbb{H}\textbf{H}^n$ and $\mathbb{O}\textbf{H}^2$, respectively. (Note that we need to consider only the cases $n\geq 2$ and when $\mathbb{K}= \mathbb{O}$, $n=2$.)\\
\indent In \cite[p. 2]{AF04}, C.S. Aravinda and F.T. Farrell ask the following:
\begin{question}\label{fjquo}
For each division algebra $\mathbb{K}$ over the reals and each integer $n\geq 2$\\ ($n = 2$ when $\mathbb{K}=\mathbb{O}$), does there exist a closed negatively curved Riemannian manifold $M^{dn}$(where $d=\dim_{\mathbb{R}} \mathbb{K}$) which is homeomorphic but not diffeomorphic to a $\mathbb{K}$-hyperbolic manifold.
\end{question}
\begin{remark}\rm
For $\mathbb{K}=\mathbb{R}$ and $n=2$, 3, this is impossible since homeomorphism implies diffeomorphism in these dimensions \cite{Moi77}.  Also when $\mathbb{K}=\mathbb{R}$, it was shown in \cite{FJO98} that the answer is yes provided $n\geq 6$. When $\mathbb{K}=\mathbb{C}$, it was shown in \cite{FJ94} that the answer is yes for $n=4m+1$ for any integer $m\geq 1$ and for $n=4$. When $\mathbb{K}=\mathbb{H}$, the answer is yes for $n = 2$, 4 and 5, see \cite{AF04}. The answer to this question is yes for $\mathbb{K}=\mathbb{O}$ by \cite{AF03} since only one dimension needs to be considered in this case. In this section, we consider the case $\mathbb{K}=\mathbb{C}$ and show that the answer is yes for $n=7,~8$.
\end{remark}
\indent Since Borel \cite{Bor63} has constructed closed complex hyperbolic manifolds in every complex dimension $m\geq 1$ and by Theorem \ref{exotho} and Theorem \ref{mainthe}, we have the following result :
\begin{theorem}\label{conmainthe}
Let $n$ be either 7 or 8. Given any positive number $\epsilon\in \mathbb{R}$, there exists a pair of closed negatively curved Riemannian manifolds $M$ and $N$ having the following properties:
\begin{itemize}
\item[\rm{(i)}] $M$ is a complex $n$-dimensional hyperbolic manifold.
 \item[\rm{(ii)}] The sectional curvatures of $N$ are all in the interval $[-4 -\epsilon, -1 +\epsilon]$.
 \item[\rm{(iii)}] The manifolds $M$ and $N$ are homeomorphic but not diffeomorphic.
 \end{itemize}
\end{theorem}
\section{ Farrell-Jones Spheres and Hitchin Sphere}
In this section, we give proofs of the Theorems \ref{condiff}, \ref{exotho} and \ref{lemhit}.
\begin{definition}\rm
Let $M$ be a topological manifold. Let $(N,f)$ be a pair consisting of a smooth manifold $N$ together with a homeomorphism $f:N\to M$. Two such pairs $(N_{1},f_{1})$ and $(N_{2},f_{2})$ are concordant provided there exists a diffeomorphism $g:N_{1}\to N_{2}$ such that the composition $f_{2}\circ g$ is topologically concordant to $f_{1}$, i.e., there exists a homeomorphism $F: N_{1}\times [0,1]\to M\times [0,1]$ such that $F_{|N_{1}\times 0}=f_{1}$ and $F_{|N_{1}\times 1}=f_{2}\circ g$. The set of all such concordance classes is denoted by $\mathcal{C}(M)$.
\end{definition}
\indent We recall some terminology from \cite{KM63} :
\begin{definition}\rm{
\begin{itemize}
\item[(a)] A homotopy $m$-sphere $\Sigma^m$ is an oriented smooth closed manifold homotopy equivalent to $\mathbb{S}^m$.
\item[(b)]A homotopy $m$-sphere $\Sigma^m$ is said to be exotic if it is not diffeomorphic to $\mathbb{S}^m$.
\item[(c)] Two homotopy $m$-spheres $\Sigma^{m}_{1}$ and $\Sigma^{m}_{2}$ are said to be equivalent if there exists an orientation preserving diffeomorphism $f:\Sigma^{m}_{1}\to \Sigma^{m}_{2}$.
\end{itemize}
The set of equivalence classes of homotopy $m$-spheres is denoted by $\Theta_m$. The equivalence class of $\Sigma^m$ is denoted by [$\Sigma^m$]. When $m\geq 5$, $\Theta_m$ forms an abelian group with group operation given by the connected sum $\#$ and the zero element represented by the equivalence class  of the round sphere $\mathbb{S}^m$. M. Kervaire and J. Milnor \cite{KM63} showed that each  $\Theta_m$ is a finite group; in particular,  $\Theta_{8}$, $\Theta_{14}$ and  $\Theta_{16}$ are cyclic groups of order $2$, $\Theta_{10}$ and $\Theta_{20}$ are cyclic groups of order 6 and 24 respectively and $\Theta_{18}$ is a group of order 16.}
\end{definition}
\indent Start by noting that there is a homeomorphism $h: M^n\#\Sigma^n \to M^n$ $(n\geq5)$ which is the inclusion map outside of homotopy sphere $\Sigma^n$ and well defined up to topological concordance. We will denote the class of $(M^n\#\Sigma^n, h)$  in $\mathcal{C}(M)$  by $[M^n\#\Sigma^n]$. (Note that $[M^n\#\mathbb{S}^n]$ is the class of $(M^n, id_{M^n})$.) Let $f_{M}:M^n\to \mathbb{S}^n $  be a degree one map. Note that $f_{M}$ is well-defined up to homotopy. Composition with $f_{M}$ defines a homomorphism
$$f_{M}^*:[\mathbb{S}^n, Top/O]\to [M^n ,Top/O],$$ and in terms of the identifications
\begin{center}
$\Theta_n=[\mathbb{S}^n, Top/O]$ and $\mathcal{C}(M^n)=[M^n ,Top/O]$
\end{center}
given by \cite[p. 25 and 194]{KS77}, $f_{M}^*$ becomes $[\Sigma^m]\mapsto [M^m\#\Sigma^m]$.
\begin{definition}
If $M$ is homotopy equivalent to $\mathbb{C}\mathbb{P}^n$, we will call a generator of $H^2(M; \mathbb{Z})$ a c-orientation of $M$.
\end{definition}
\indent Hereafter $g$ denotes the conjugation map $$(z_0,z_1,z_2,z_3,z_4,...,z_n)\mapsto (\bar{z}_0,\bar{z}_1,\bar{z}_{2},\bar{z}_{3},\bar{z}_{4},...,\bar{z}_n)$$ (the complex conjugation) induces the diffeomorphism $g:\mathbb{C}\mathbb{P}^{n} \to \mathbb{C}\mathbb{P}^{n}$ such that $g^*(c_1)=-c_1$ where $c_1$ is the $c$-orientation of $\mathbb{C}\mathbb{P}^{n}$.\\
\paragraph{Proof of Theorem \ref{condiff}:}
Assume that $\Sigma^{2n}$ is a Farrell-Jones sphere. Suppose $\mathbb{C}\mathbb{P}^{n}\#\Sigma^{2n}$ and $\mathbb{C}\mathbb{P}^{n}$ are diffeomorphic. If $f: \mathbb{C}\mathbb{P}^{n}\#\Sigma^{2n}\to \mathbb{C}\mathbb{P}^{n}$ is a diffeomorphism, then $f$ induces an isomorphism on cohomology $$f^*:H^{*}(\mathbb{C}\mathbb{P}^{n},\mathbb{Z})\to H^{*}(\mathbb{C}\mathbb{P}^{n}\#\Sigma^{2n},\mathbb{Z})$$
such that $f^*(c_1)=\pm c_2$, where $c_1$ and $c_2$ are  $c$-orientation of $\mathbb{C}\mathbb{P}^{n}$ and $\mathbb{C}\mathbb{P}^{n}\#\Sigma^{2n}$ respectively. If $f^*(c_1)=c_2$, then $f$ is a c-orientation preserving diffeomorphism. If $f^*(c_1)=-c_2$, then $g\circ f$ is a c-orientation preserving diffeomorphism, where  $g:\mathbb{C}\mathbb{P}^{n} \to \mathbb{C}\mathbb{P}^{n}$ is the conjugation map. In both cases, we have that $\mathbb{C}\mathbb{P}^{n}\#\Sigma^{2n}$ is c-orientation diffeomorphic to $\mathbb{C}\mathbb{P}^{n}$. By \cite[Corollary 3, p. 97]{Sul67}, $\mathbb{C}\mathbb{P}^{n}\#\Sigma^{2n}$ is concordant to $\mathbb{C}\mathbb{P}^{n}$. This is a contradiction since $\Sigma^{2n}$ is a Farrell-Jones sphere. Thus $\mathbb{C}\mathbb{P}^{n}\#\Sigma^{2n}$ and $\mathbb{C}\mathbb{P}^{n}$ are not diffeomorphic. Conversely, suppose $\mathbb{C}\mathbb{P}^{n}\#\Sigma^{2n}$ and $\mathbb{C}\mathbb{P}^{n}$ are not diffeomorphic. Then, by \cite[Corollary 
3, p. 97]{Sul67}, $\mathbb{C}\mathbb{P}^{n}\#\Sigma^{2n}$ is not concordant to $\mathbb{C}\mathbb{P}^{n}$. This shows that $\Sigma^{2n}$ is a Farrell-Jones sphere. This completes the proof of Theorem \ref{condiff}. ~~~~~~~~~~~~~~~~~~~~~~~~~~~~~~~~~~~~~~~~~~~~~~~~~~~~~~~~~~~~~~~~~~~~~~~~~~~~~~~~~  $\Box$
\paragraph{Proof of Theorem \ref{exotho}:}
Let $\Sigma^{2n}$ be the generator of $\Theta_{2n}$ $({\rm{with}} n=7~or~8)$. Suppose $\Sigma^{2n}$ is not a Farrell-Jones sphere. Then $\mathbb{C}\mathbb{P}^{n}\#\Sigma^{2n}$ is concordant to $\mathbb{C}\mathbb{P}^{n}$. By \cite[Corollary 3, p. 97]{Sul67}, $\mathbb{C}\mathbb{P}^{n}\#\Sigma^{2n}$ is c-orientation diffeomorphic to $\mathbb{C}\mathbb{P}^{n}$. Let $f: \mathbb{C}\mathbb{P}^{n}\#\Sigma^{2n}\to \mathbb{C}\mathbb{P}^{n}$ be a c-orientation diffeomorphism such that $f^*(c_1)=c_2$, where $c_1$ and $c_2$ are  $c$-orientation of $\mathbb{C}\mathbb{P}^{n}$ and $\mathbb{C}\mathbb{P}^{n}\#\Sigma^{2n}$ respectively. Using properties of the cup product, we have that $f^*(c_1^n)=c_2^n$. If $c_1=c_2$ in  $H^{2}(\mathbb{C}\mathbb{P}^{n},\mathbb{Z})$, then $f$ is an orientation preserving diffeomorphism. If $c_1\neq c_2$ in $H^{2}(\mathbb{C}\mathbb{P}^{n},\mathbb{Z})$, then $g\circ f$ is an orientation preserving diffeomorphism with the property that $(g\circ f)^*(c_1)=f^*(g^*(c_1))=-c_2=c_1,$, where  $g:\mathbb{C}\mathbb{P}^{n} \to \mathbb{C}\mathbb{P}^{n}$ is the conjugation map. In both cases, we have that $\mathbb{C}\mathbb{P}^{n}\#\Sigma^{2n}$ is an orientation preserving diffeomorphic to $\mathbb{C}\mathbb{P}^{n}$. This is a contradiction because, by \cite[Theorem 1]{Kaw68}, $\mathbb{C}\mathbb{P}^{n}\#\Sigma^{2n}$ can not be  orientation preserving diffeomorphic to $\mathbb{C}\mathbb{P}^{n}$. Thus $\Sigma^{2n}$ is a Farrell-Jones sphere. This completes the proof of Theorem \ref{exotho}. ~~~~~~~~~~~~~~~~~~~~~~~~~~~~~~~~~~~~~~~~~~~~~~~~~~~~~~$\Box$
\indent Recall that the $\alpha$-invariant is the ring homomorphism $\alpha:\Omega_{*}^{spin}\to KO_{*}$ which associates to a spin bordism class the $KO$-valued index of the Dirac operator of a representative spin manifold. We also write $\alpha$ for the corresponding invariant on framed bordism :
\begin{center}
 $\alpha :\Omega_{*}^{f}\to \Omega_{*}^{spin}\to KO_{*}$
\end{center}
Under the Pontryagin-Thom isomorphism $\Omega_{*}^{f}\cong \pi_{*}^{s}$, the $\alpha$-invariant has the following interpretation as Adams $d$-invariant $d_{\mathbb{R}} : \pi_{r}^{s}\to KO_{*}$, which was used already in \cite[p. 44]{Hit74} and \cite[Lemma 2.12]{CS12}.
\begin{lemma}\label{adaminva}
 Under the Pontryagin-Thom isomorphism $\Omega_{*}^{f}\cong\pi_{*}^{s}$, the $\alpha$-invariant $\alpha:\Omega_{8n+2}^{f}\to KO_{8n+2}$
 may be identified with $d_{\mathbb{R}} : \pi_{8n+2}^{s}\to KO_{8n+2}$.
\end{lemma}
We start by recalling some facts from smoothing theory \cite{Bru71a}, which were used already in \cite[Lemma 3.17]{FJ94}. There are $H$-spaces $SF$, $F/O$ and $Top/O$ and $H$-space maps $\phi:SF\to F/O$,  $\psi: Top/O\to F/O$ such that
\begin{equation}\label{eq1.bru}
\psi_{*}:\Theta_{8n+2}=\pi_{8n+2}(Top/O)\to \pi_{8n+2}(F/O)
\end{equation}
is an isomorphism for $n\geq 1$. The homotopy groups of $SF$ are the stable homotopy groups of spheres $\pi^{s}_m$ ; i.e., $\pi_{m}(SF)=\pi^{s}_m$ for $m\geq 1$. For $n\geq 1$,
\begin{equation}\label{eq2.bru}
 \phi_{*}:\pi^{s}_{8n+2}\to \pi_{8n+2}(F/O)
\end{equation}
is an isomorphism. Since every homotopy sphere has a unique spin-structure, we obtain the $\alpha$-invariant on $\pi^{s}_{8n+2} \cong \pi_{8n+2}(F/O) \cong \Theta_{8n+2}$ :
$$\alpha : \pi^{s}_{8n+2}\stackrel{\phi_{*}}{\rightarrow} \pi_{8n+2}(F/O) \stackrel{\psi_{*}^{-1}}{\rightarrow}\Theta_{8n+2}\to \Omega_{8n+2}^{spin}\to KO_{8n+2},$$ where $\psi_{*}$ and $\phi_{*}$ are the isomorphisms as in Equation (\ref{eq1.bru}) and (\ref{eq2.bru}) respectively.\\
\indent Let $\rm{Ker}(d_{\mathbb{R}})$ denotes the kernel of the Adams $d$-invariant $d_{\mathbb{R}} : \pi_{8n+2}^{s}\to \mathbb{Z}_2$. By Lemma \ref{adaminva}, $\rm{Ker}(d_{\mathbb{R}})$ consists of framed manifolds which bound spin manifolds.\\
\paragraph{Proof of Theorem \ref{lemhit}:}
Consider the following commutative of diagram :
\begin{equation}\label{digram1}
\begin{CD}
[\mathbb{S}^{2m},Top/O]=\Theta_{2m}@>f^*_{\mathbb{C}\mathbb{P}^{m}}>> [\mathbb{C}\mathbb{P}^{m},Top/O]=\mathcal{C}(\mathbb{C}\mathbb{P}^{m})\\
@VV\psi_{*}V                      @VV\psi_{*}V\\
[\mathbb{S}^{2m},F/O]  @>f^*_{\mathbb{C}\mathbb{P}^{m}}>> [\mathbb{C}\mathbb{P}^{m},F/O]\\
@AA\phi_{*}A                      @AA\phi_{*}A\\
[\mathbb{S}^{2m},SF]=\pi_{2m}^{s} @>f^*_{\mathbb{C}\mathbb{P}^{m}}>>  [\mathbb{C}\mathbb{P}^{m},SF]
\end{CD}
\end{equation}
In this diagram, $\phi_{*}$ and $\psi_{*}$ are induced by the $H$-space maps $\phi:SF\to F/O$,  $\psi: Top/O\to F/O$ respectively and the homomorphism $$\phi_{*}:[\mathbb{C}\mathbb{P}^m,SF]\to [\mathbb{C}\mathbb{P}^m,F/O]$$ is monic for all $m\geq 1$ by a result of Brumfiel \cite[p. 77]{Bru71}. Recall that the concordance class $[\mathbb{C}\mathbb{P}^m\#\Sigma]\in [\mathbb{C}\mathbb{P}^m,Top/O]$ of $\mathbb{C}\mathbb{P}^m\#\Sigma$ is $f^{*}_{\mathbb{C}\mathbb{P}^m}([\Sigma])$ when $m > 2$, and that $[\mathbb{C}\mathbb{P}^m]=[\mathbb{C}\mathbb{P}^m\#\mathbb{S}^{2m}]$ is the zero element of this group.\\
Let $\Sigma^{8n+2}\in \Theta_{8n+2}$ be a Hitchin $(8n+2)$-sphere $({\rm{with}} n\geq 1)$ and further let $\eta \in \pi^{s}_{8n+2}=[\mathbb{S}^{8n+2},SF]$ be such that $$\psi_{*}^{-1}(\phi_{*}(\eta))=\Sigma^{8n+2}.$$ Recall that $[X,~SF]$ can be identified with the $0^{th}$ stable cohomotopy group $\pi^0(X)$. Let $h:\mathbb{S}^{q+8n+2}\to \mathbb{S}^{q}$ represent $\eta$. Since $\Sigma^{8n+2}$ is a Hitchin sphere and by Lemma \ref{adaminva}, we have $$0 \neq \alpha(\Sigma^{8n+2})=d_{\mathbb{R}}(h)=h^*\in {\rm Hom}(\widetilde{KO}^q(\mathbb{S}^{q}),\widetilde{KO}^q(\mathbb{S}^{q+8n+2})).$$ Also Adams and Walker \cite{AW65} showed that $\Sigma^{q}f_{\mathbb{C}\mathbb{P}^{4n+1}}:\Sigma^{q}\mathbb{C}\mathbb{P}^{4n+1}\to \mathbb{S}^{q+8n+2}$ induces a monomorphism on $\widetilde{KO}^q(.)$. Consequently the composite map
\begin{equation*}\label{eq8.bru}
h\circ \Sigma^{q}f_{\mathbb{C}\mathbb{P}^{4n+1}}: \Sigma^{q}\mathbb{C}\mathbb{P}^{4n+1}\to \mathbb{S}^{q}
\end{equation*}
induces a non-zero homomorphism on $\widetilde{KO}^q(.)$.  This shows that $$f^{*}_{\mathbb{C}\mathbb{P}^{4n+1}}(\eta)=[h\circ \Sigma^{q}f_{\mathbb{C}\mathbb{P}^{4n+1}}]\neq 0,$$ where $$f^{*}_{\mathbb{C}\mathbb{P}^{4n+1}}: [\mathbb{S}^{8n+2},SF]\to [\mathbb{C}\mathbb{P}^{4n+1},SF].$$
Since the homomorphism $\phi_{*}:[\mathbb{C}\mathbb{P}^m,SF]\to [\mathbb{C}\mathbb{P}^m,F/O]$ is monic, by the commutative diagram (\ref{digram1}) where $m=4n+1$, we have $$\psi_{*}( f^*_{\mathbb{C}\mathbb{P}^{4n+1}}(\Sigma^{8n+2}))=\phi_{*}(f^{*}_{\mathbb{C}\mathbb{P}^{4n+1}}(\eta))\neq 0.$$ This implies that $$f^*_{\mathbb{C}\mathbb{P}^{4n+1}}(\Sigma^{8n+2})\neq0$$ and hence $\mathbb{C}\mathbb{P}^{4n+1}\#\Sigma^{8n+2}$ is not concordant to $\mathbb{C}\mathbb{P}^{4n+1}$. This shows that $\Sigma^{8n+2}$ is a Farrell-Jones sphere and this completes the proof of Theorem \ref{lemhit}. ~~~~~~~~~~~~~~~~~~~~~~~~~~~~~~~~~~~~~~~~~~~~~~~~~~~~~~~~~~~~~~~~~~~~~~~~~~~~~~~~~~~~~~~~~~~   $\Box$\\~\\
\begin{remark}\label{rem.ander}\rm{
 \indent
\begin{itemize}
\item[\bf{1.}] Let us note that the homotopy sphere $\Sigma^{8n+2}$ $(n\geq 1)$ given by \cite[Lemma 3.17]{FJ94} is the image of Adams element $\mu_{8n+2}$ of order 2 under the composed isomorphism $\psi_{*}^{-1}\circ \phi_{*}$,  where $\psi_{*}$ and $\phi_{*}$ are the isomorphisms as in Equation (\ref{eq1.bru}) and (\ref{eq2.bru}) respectively (see \cite[Equation (3.17.4)]{FJ94}). By \cite[Theorem 1.2]{Ada66} and Lemma \ref{adaminva}, we have $$d_{\mathbb{R}}(\mu_{8n+2})=\alpha(\Sigma^{8n+2})=1.$$  This shows that  $\Sigma^{8n+2}$ is a Hitchin sphere of order 2 in $\Theta_{8n+2}$. By Theorem \ref{lemhit}, $\Sigma^{8n+2}$ is a Farrell-Jones sphere.
\item[\bf{2.}] Since $\Theta_{18}\cong \rm{Ker}(\alpha)\oplus \mathbb{Z}_2$, where the $\alpha$-invariant $\alpha: \Theta_{18}\to \mathbb{Z}_2$ satisfies $\rm{Ker}(\alpha)=\mathbb{Z}_8$ (see \cite[p. 12]{CS12}). This shows that there are exotic spheres of order $\neq 2$ in $\Theta_{18}$ which are not in the kernel of $\alpha$. This implies that there are Hitchin spheres of order $\neq 2$ in $\Theta_{18}$ which are all Farrell-Jones sphere by Theorem \ref{lemhit}.
\item[\bf{3.}] In \cite{ABP67}, Anderson, Brown and Peterson proved that one has $\alpha(\Sigma^m)\neq 0$ iff $m=8k+1$ or $8k+2$ iff $\Sigma^m$ is an exotic sphere not bounding a spin manifold, where $\alpha:\Theta_m\to \Omega_m^{spin}\to KO_{m}$ is the $\alpha$-invariant. This implies that $\Sigma^m$ is a Hitchin sphere in $\Theta_m$ iff $\Sigma^m$ is an exotic sphere not bounding a spin manifold. By Theorem \ref{lemhit}, every exotic sphere not bounding a spin manifold $\Sigma^{8n+2}$ in $\Theta_{8n+2}$ is a Farrell-Jones sphere.
\item [\bf{4.}] By \cite[Theorem 7.2]{Ada66}, $\Theta_{10}\cong \rm{Ker}(d_{\mathbb{R}})\oplus \mathbb{Z}_2$ such that $\rm{Ker}(d_{\mathbb{R}})=\mathbb{Z}_3$. If $\Sigma^{10}$ is a generator of $\rm{Ker}(d_{\mathbb{R}})$, then $d_{\mathbb{R}}(\Sigma^{10})=\alpha(\Sigma^{10})=0$. This shows that $\Sigma^{10}$ is not a Hitchin sphere. But, by \cite[Lemma 3.17]{FJ94}, $\Sigma^{10}$ is a Farrell-Jones sphere.
\end{itemize}}
\end{remark}
\section{ The Inertia Groups of Complex Projective Spaces}
In this section, we discuss the relationship between inertia groups of $\mathbb{C}\mathbb{P}^n$ and Farrell-Jones spheres.
\begin{definition}\rm{
Let $M^m$ be a closed smooth, oriented $m$-dimensional manifold. Let $\Sigma\in \Theta_{m}$ and let $g:\mathbb{S}^{m-1}\to \mathbb{S}^{m-1}$ be an orientation preserving diffeomorphism corresponding to $\Sigma$. Writing $M\#\Sigma$ as $(M^m\setminus \rm{int}(\mathbb{D}^m))\cup_{g}\mathbb{D}^m$, let $\iota : M\to M\#\Sigma$ denote the PL homeomorphism defined by $\iota_{|M\setminus \rm{int}(\mathbb{D}^m)}=id$ and $\iota_{|\mathbb{D}^m}=Cg$, where $Cg:\mathbb{D}^m\to \mathbb{D}^m$ is the cone extension of $g$.\\
The inertia group $I(M)\subset \Theta_{m}$ is defined as the set of $\Sigma \in \Theta_{m}$ for which there exists an orientation preserving diffeomorphism $\phi :M\to M\#\Sigma$.\\
Define the homotopy inertia group $I_h(M)$ to be the set of all $\Sigma\in I(M)$ such that there exists a diffeomorphism $M\to M\#\Sigma$ which is homotopic to $\iota$.\\
Define the concordance inertia group $I_c(M)$ to be the set of all $\Sigma\in I_h(M)$ such that $M\#\Sigma$ is concordant to $M$. Clearly, $I_c(M)\subseteq I_h(M)\subseteq I(M)$. Note that for $M=\mathbb{C}\mathbb{P}^{n}$, Theorem \ref{condiff} can be restated as : }
\end{definition}
\begin{theorem}\label{con-iner}
A sphere $\Sigma^{2n}\in \Theta_{2n}$ is a Farrell-Jones sphere iff $\Sigma^{2n}\notin I(\mathbb{C}\mathbb{P}^{n})$.
\end{theorem}
\begin{remark}\rm
 Since $I_c(\mathbb{C}\mathbb{P}^{n})\subseteq I_h(\mathbb{C}\mathbb{P}^{n})\subseteq I(\mathbb{C}\mathbb{P}^{n})$ and by the above Theorem \ref{con-iner}, we have that $$I_c(\mathbb{C}\mathbb{P}^{n})=I_h(\mathbb{C}\mathbb{P}^{n})=I(\mathbb{C}\mathbb{P}^{n}).$$
\end{remark}
The proof of theorem \ref{lemhit} leads one to the following question:
\begin{question}\label{raquestion}
 Let $f:\mathbb{C}\mathbb{P}^{4n+1}\to \mathbb{S}^{8n+2}$ be any degree one map $(n\geq 1)$.\\
 Does there exist an element $\eta\in \rm{Ker}(d_{\mathbb{R}})\subset \pi^s_{8n+2}=\Theta_{8n+2}$ such that the following is true :
 \begin{itemize}
\item [$(\star)$] If any map $h:\mathbb{S}^{q+8n+2}\to \mathbb{S}^{q}$ represents $\eta$, then $$h\circ \Sigma^{q}f : \Sigma^{q}\mathbb{C}\mathbb{P}^{4n+1}\to \mathbb{S}^{q}$$ is not null homotopic.
 \end{itemize}
\end{question}
\begin{remark}\label{rem1}\rm{
\indent
 \begin{itemize}
\item [\bf{1.}] By \cite[Lemma 9.1]{Kaw69}, $I(\mathbb{C}\mathbb{P}^{4n+1})\subseteq\rm{Ker}(d_{\mathbb{R}})$. If the answer to Question \ref{raquestion} is yes, then we have $I(\mathbb{C}\mathbb{P}^{4n+1})\neq \rm{Ker}(d_{\mathbb{R}})$,~ i.e., there exists an exotic sphere $\Sigma$ bounding spin manifold in $\Theta_{8n+2}$ such that $\Sigma \notin I(\mathbb{C}\mathbb{P}^{4n+1})$. This can be seen as follows :
Let $\eta\in \rm{Ker}(d_{\mathbb{R}})$ and let $h:\mathbb{S}^{q+8n+2}\to \mathbb{S}^{q}$ represent $\eta$ such that $h\circ \Sigma^{q}f : \Sigma^{q}\mathbb{C}\mathbb{P}^{4n+1}\to \mathbb{S}^{q}$ is not null homotopic. This implies that $$f^*_{\mathbb{C}\mathbb{P}^{4n+1}}(h)=[h\circ \Sigma^{q}f_{\mathbb{C}\mathbb{P}^{4n+1}}]\neq0,$$ where $f^*_{\mathbb{C}\mathbb{P}^{4n+1}}:\pi^0(\mathbb{S}^{8n+2})\to \pi^0(\mathbb{C}\mathbb{P}^{4n+1})$. A similar argument given in the proof of Theorem \ref{lemhit} using the commutative diagram (\ref{digram1}) shows that there exists an exotic sphere $\Sigma\in \Theta_{8n+2}$ such that $\psi_{*}^{-1}\circ \phi_{*}(\eta)=\Sigma$, $d_{\mathbb{R}}(\eta)=\alpha(\Sigma)=0$ and $\mathbb{C}\mathbb{P}^{4n+1}\#\Sigma$ is not concordant to  $\mathbb{C}\mathbb{P}^{4n+1}$, where $\psi_{*}$ and $\phi_{*}$ are the isomorphisms as in Equation (\ref{eq1.bru}) and (\ref{eq2.bru}) respectively. This implies that $\Sigma$ is a Farrell-Jones sphere such that $\Sigma\in \rm{Ker}(d_{\mathbb{R}})$. 
By Theorem \ref{con-iner}, $I(\mathbb{C}\mathbb{P}^{4n+1})\neq \rm{Ker}(d_{\mathbb{R}})$.
\item [\bf{2.}] If all non-zero elements in $\rm{Ker}(d_{\mathbb{R}})$ satisfy the condition $(\star)$ in Question \ref{raquestion}, then, by the above remark (1), $\Sigma\notin I(\mathbb{C}\mathbb{P}^{4n+1})$ for all exotic sphere $\Sigma\in \rm{Ker}(d_{\mathbb{R}})$ and hence $I(\mathbb{C}\mathbb{P}^{4n+1})=0$.
\end{itemize}}
\end{remark}
\begin{theorem}
 Let $n$ be a positive integer such that $\Theta_{8n+2}$ is a cyclic group of order 2.
 Then $I(\mathbb{C}\mathbb{P}^{4n+1})=0$.
\end{theorem}
\begin{proof}
Let $\Sigma^{8n+2}$ be the generator of $\Theta_{8n+2}\cong \mathbb{Z}_2$. Let $$\psi_{*}:\Theta_{8n+2}\to \pi_{8n+2}(F/O)~~~ {\rm{and}}~~~  \phi_{*}:\pi^{s}_{8n+2}\to \pi_{8n+2}(F/O)$$ be the isomorphisms as in Equation (\ref{eq1.bru}) and (\ref{eq2.bru}). By \cite[Theorem 1.2]{Ada66}, there exists an element $\mu_{8n+2}$ of order 2 in $\pi^{s}_{8n+2}$. This shows that $$\phi_{*}^{-1}\circ \psi_{*}(\Sigma^{8n+2})=\mu_{8n+2}.$$ By \cite[Theorem 1.2]{Ada66} and Lemma \ref{adaminva}, we have $d_{\mathbb{R}}(\mu_{8n+2})=\alpha(\Sigma^{8n+2})=1$. This implies that $\Sigma^{8n+2}$ is a Hitchin sphere. By Theorem \ref{lemhit} and Theorem \ref{condiff}, $\mathbb{C}\mathbb{P}^{4n+1}\#\Sigma^{8n+2}$ is not diffeomorphic to $\mathbb{C}\mathbb{P}^{4n+1}$. 
This implies that\\ $I(\mathbb{C}\mathbb{P}^{4n+1})=0$.
\end{proof}

\end{document}